\documentclass[11pt,a4paper]{article}
\usepackage{amssymb,amsmath,amsthm}
\usepackage[english]{babel}
\usepackage{t1enc}
\usepackage[latin2]{inputenc}
\usepackage{graphicx}
\usepackage{footnpag}
\usepackage{url}

\makeatletter
\newtheorem*{rep@theorem}{\rep@title}
\newcommand{\newreptheorem}[2]{%
\newenvironment{rep#1}[1]{%
 \def\rep@title{#2 \ref{##1}}%
 \begin{rep@theorem}}%
 {\end{rep@theorem}}}
\makeatother

\newtheorem{thm}{Theorem}
\newreptheorem{thm}{Theorem}
\newtheorem{cor}[thm]{Corollary}
\newtheorem{defi}[thm]{Definition}
\newtheorem{lem}[thm]{Lemma}
\newtheorem{prop}[thm]{Proposition}
\newtheorem{claim}[thm]{Claim}

\newtheorem{obs}[thm]{Observation}

\usepackage{indentfirst}

\newcommand{\VR}{\textit{VR}}
\newcommand{\eps}{\varepsilon}
\DeclareMathOperator{\dist}{dist}

\def\N{\mathbb N}
\def\Z{\mathbb Z}

\newcommand{\A}{\ensuremath{\mathbf{A}}}
\newcommand{\B}{\ensuremath{\mathbf{B}}}

\begin{document}

\title{Advantage in the discrete Voronoi game}
\author{D\'aniel Gerbner\thanks{Alfr\'ed R\'enyi Institute of Mathematics, Hungarian Academy of Sciences, Budapest, Hungary; \protect\url {gerbner.daniel@renyi.mta.hu}. Research supported by Hungarian Science Foundation EuroGIGA Grant OTKA NN 102029}, Viola M\'esz\'aros\thanks
{Bolyai Institute, University of Szeged,
Aradi v\'ertan\'uk tere 1, 6720 Szeged, Hungary;
\protect\url
{viola@math.u-szeged.hu}.
Supported by OTKA Grant K76099, OTKA Grant 102029 and by the ESF EUROCORES programme EuroGIGA-ComPoSe,
  Deutsche Forschungsgemeinschaft (DFG), under grant FE 340/9-1.},  
D\"om\"ot\"or P\'alv\"olgyi\thanks{
Institute of Mathematics, E\"otv\"os University, Budapest, Hungary;
\protect\url
{dom@cs.elte.hu}.
Supported by Hungarian National Science Fund (OTKA), under grant PD 104386 and under grant NN 102029 (EUROGIGA project GraDR 10-EuroGIGA-OP-003) and the J\'anos Bolyai Research Scholarship of the Hungarian Academy of Sciences.}, 
Alexey Pokrovskiy\thanks
{Department of Mathematics, London School of Economics and Political Science, London, UK;
\protect\url
{a.pokrovskiy@lse.ac.uk}.
Supported by the LSE postgraduate Research Studentship Scheme},
G\"unter Rote\thanks
{Freie Universit\"at Berlin, Institut f\"ur Informatik,
Takustr.~9, 14195 Berlin, Germany;
\protect\url
{rote@inf.fu-berlin.de}.
Supported by the
 ESF EUROCORES programme EuroGIGA-VORONOI, Deutsche Forschungsgemeinschaft (DFG): grant RO 2338/5-1.}
}

\maketitle

\begin{abstract}
We study the discrete Voronoi game, where two players alternately
claim vertices of a graph for $t$ rounds.  In the end, the remaining
vertices are divided such that 
each player receives the vertices that
are closer to his or her claimed vertices.  We prove that there are
graphs for which the second player gets almost all vertices in this
game, but this is not possible for bounded-degree graphs.  For trees,
the first player can get at least one quarter of the vertices, and we
give examples where she can get only little more than one third of
them.
We make some general observations, relating the result with many rounds to
the result for the one-round game on the same graph.  
\end{abstract}

\section{Introduction}
The classic facility location problem deals with finding the optimal location for a facility (such as a supermarket, hospital, fire station) with respect to a given set of customers.  Typically, we want to place the facility to minimize the distance customers need to travel to get to it.
Competitive facility location is a variant of the problem when  several service providers compete for the interests
of the same set of customers.  An example would be two supermarket
chains building shops in a city -- with each chain trying to attract
the largest number of customers.  

We study a simple model of competitive facility location called the
\emph{Voronoi game}.  
This game 
is a game played on a measurable
metric space by two players. 
 The players alternate in
 placing a facility on a single point in the space
 The game lasts for a fixed number of rounds.  At the end of the
game, the space is divided between the two players: each player
receives the area which is closer to his or her facilities,
or in other words, the sum of the areas of the corresponding regions in
the Voronoi diagram.  The winner is
the player who controls the greater portion of the space.

The Voronoi game was first defined
 by Ahn, Cheng, Cheong, Golin, and
van Oostrum~\cite{ACCGO}, who studied it on lines and circles.
Subsequently a discrete version of the game emerged,
on which we shall focus in this paper; for results on the continuous game see e.g.\ \cite{CHLM, FM}.  The
\emph{discrete Voronoi game} is played on the vertices of a graph $G$
by two players called \A\ and \B\ for a fixed number $t$ of rounds.
Player \A\ starts, and they alternatingly claim vertices of $G$ during
each round $1, \dots, t$.  No vertex may be claimed more than once.
At the end of the game, the remaining vertices are divided between the
players -- with each player receiving the vertices that are closer to
his or her claimed vertices. If a vertex is equidistant to each players' claimed vertices then it is split evenly between \A\ and \B\ (each player receives half a vertex.) The winner is the player who controls more vertices at the end.

This natural variant was first studied by Demaine, Teramoto and
Uehara~\cite{TDU}, who showed that it is NP-complete to determine the
winner in the Voronoi game on a general graph $G$, even if the game
lasts for only one round, (but player \B\ can place more than one pebble).   They also studied the game on a large
$k$-ary tree and showed that under optimal play, the first player wins
if $k$ is odd, and that the game ends in a tie when $k$ is even.  The
game on trees was studied further by Kiyomi, Saitoh, and Uehara~\cite{KSU} who
completely solved the game on a path -- they showed that the game on a
path with $n$ vertices played for $t<n/2$ rounds always ends in a draw,
unless $n$ is odd and $t=1$, in which case \A\ wins (by having one vertex more).
There are many results that deal with various algorithmic questions
about variations and special cases of the Voronoi game, for example for
 weighted graphs \cite{bbd-os1rd-11},
in a planar geometric setting 
 \cite{ceh-fgsms-07,bn-nvmcf-13}, or
for a ``continuous'' graph model~\cite{bbds-vgg-13}.
%


The above results suggest that in general it is hard to
determine the winner of the Voronoi game on a graph.  Therefore, in
this paper, we will not be concerned with deciding the winner of the
game -- rather we are interested in knowing how large proportion of the vertices a player can control at the end of the game. One particular question we are interested in is the following: ``for $\epsilon>0$ and a graph with $n$ vertices, does $A$ (or $B$) have a strategy to control at least $\epsilon n$ vertices at the end of the game?'' This
question motivates the following definition.

\begin{defi}
For a given graph $G$ define its {\em Voronoi ratio}, $\VR(G,t)$, as the number of vertices that belong to \A, plus half of the number of tied vertices (if there are any) divided by the total number of vertices in $G$ after an optimal play of $t$ rounds.
\end{defi}

It is not immediately clear what range $\VR(G,t)$ can take.  By considering a star $S_k$ with $k$ leaves, it is easy to show that
\begin{equation}\label{star}
\VR(S_k,t)=1-\frac{t}{k+1}.
\end{equation}
This shows that $\VR(G,t)$ can be arbitrarily close to $1$, and hence
\A\ can control almost all the vertices by the end of the game.  By
considering a path it is possible to show that $\VR(G,t)=1/2$ is
possible as well~\cite{KSU}.  However, constructing a graph which
satisfies $\VR(G,1)<1/2$ is already non-trivial.  The smallest such
graph that we know of has $9$ vertices.
It consists of a cycle of length six, with an additional leaf attached to every other vertex of the cycle.  It is easy to check that at the end of the 1-round Voronoi game on this graph\ \B\ can always control 5 of the 9 vertices.  In Section~\ref{NoLowerBoundSection} we show that, in fact, $\VR(G,t)$ can be arbitrarily close to zero.

\begin{thm}\label{NoLowerBound}
For every $\epsilon>0$ and $t \in \N$, there is a graph $G$ with $\VR(G,t)<\epsilon$.
\end{thm}

This theorem, together with (\ref{star}) shows that in general the discrete Voronoi game does not favor either player.\footnote{In fact the construction easily generalizes to the continuous Voronoi game as well, but here we focus only on the discrete version.}  However there may be natural classes of graphs on which one of the players has a significant advantage.  

In Section~\ref{TreesSection}, we study the Voronoi game on a tree and
show that every tree $T$ satisfies $\VR(T,t)\geq 1/4$ for all $t$.
When the number of rounds is small, the first player may obtain an
even larger advantage.  It was noted in~\cite{TDU} that $\VR(T,1)\geq
1/2$ for every tree $T$.  We show that $\VR(T,2)\geq 1/3$, for any
tree $T$ and construct trees whose Voronoi ratio is arbitrarily close
to $1/3$ for $t\ge 2$ moves.

In Section~\ref{BoundedDegreeSection}, we study the Voronoi game on a
graph with bounded maximum degree.  We show that every graph $G$ with
maximum degree $\Delta$ has $\VR(G,t)\leq 1-1/2\Delta$.  We show that the bound
in this result cannot be decreased 
substantially 
 by constructing graphs $G$  with maximum degree $\Delta$ whose Voronoi ratio is arbitrarily close to $1-1/\Delta$.

In order to prove some of the above results, we first establish bounds on the Voronoi ratio which hold for \emph{all} graphs.  In Section~\ref{GeneralBoundsSection} show that for any $t$, $\VR(G,t)$ can be bounded in terms of the quantity $\VR(G,1)$:

\begin{thm}\label{GeneralBound}
For every graph $G$ and $t\geq 1$ we have 
$$\frac{1}{2}\VR(G,1)\leq \VR(G,t)\leq \frac{1}{2}(\VR(G,1)+1).$$
\end{thm}
Thus, 
to a limited extent, the
outcome of the Voronoi game is determined just by the outcome of the
one-round game.
In particular, if the Voronoi ratio for one round it
close to 1, then it cannot be close to 0 for more rounds, and vice versa.
This theorem is 
useful for finding good bounds on the
Voronoi ratio of various classes of graphs beyond those considered in
this paper.


\section{General bounds on $\VR(G,t)$}\label{GeneralBoundsSection}
In this section we give  bounds for $\VR(G,t)$ for a  graph $G$ in terms of $\VR(G,1)$.
We prove Theorem~\ref{GeneralBound}.

\begin{proof}
Both inequalities 
are proved by strategy stealing arguments. Let $n=|V(G)|$.

First we prove the left-hand inequality, $\VR(G,1)/2\leq \VR(G,t)$.
Suppose that $\B$ has a strategy in the $t$-round game that gives him more than
$1-\VR(G,1)/2$ of vertices in $G$.

  Let $v$ be the optimal vertex to pick for \A\ in the one-round game.
  Player \A's strategy for the $t$-round game is as follows.  First
  she picks $v$.  Then she pretends that she has not picked it and
  follows $\B$'s strategy, which would give her a fraction $1- \VR(G,1)/2$,
  except that she cannot play the last move.  The vertices that
  \A\ could have controlled by playing the last move $u$, but doesn't
  control having played $v$ are contained in $S=\{x\in G: \mathrm{dist}(x,u)\leq
  \mathrm{dist}(x,v)\}$.  By the definition of $\VR(G,1)$, we have $|S|\leq( 1-
  \VR(G,1))n$.  So at the end of the game $\A$ controls at least 
 $(1-\VR(G,1)/2)n-|S|\ge \VR(G,2)n/2$ vertices, proving the lower bound. 

Now we prove the right-hand inequality, 
$\VR(G,t)\leq \frac{1}{2}(\VR(G,1)+1)$.
Suppose $A$ plays $v_A$ in her first move, and
let $v_B$ be the best response of \B\ if he were playing the one-round
game. 
Let $H=\{\,h\in G: \mathrm{dist}(h,v_B)< \mathrm{dist}(h,v_A)\,\}$  and $K=\{\,k\in G: \mathrm{dist}(k,v_B)= \mathrm{dist}(k,v_A)\,\}$.  By definition of
$\VR(G,1)$, we have that $|H|+|K|/2\geq (1-\VR(G,1))n$.


For the remainder of the game \B\ is only interested in controlling as much of $H\cup K$ as possible.
In order to do this, we consider an auxiliary game called \emph{the new game} played on the graph $G-v_A$.  The following are the rules of {the new game}.
\begin{itemize}
\item Two players, named X and Y, alternate.  Player X goes first.
\item The game lasts for $t-1$ rounds.
\item Before the start of play, the vertex $v_B$ is occupied by player Y.
\item At the end of the game, the players score a point for each
  vertex of $H$ that they control and half a point for each vertex of
  $K$ that they control. 
Accordingly, tied vertices in $H$ give half a point to each player and tied vertices in $K$ give a quarter point to each player.  The winner is the player with the most points.
\end{itemize}
The winner of the new game scores at least $|H|/2+|K|/4$ points.  We
will show that \B\ can always end up controlling at least $|H|/2+|K|/4$
vertices at the end of the original game.  This proves the upper bound
of the theorem since $|H|/2+|K|/4\geq (1-\VR(G,1))|G|/2$.

  Player \B's strategy in the original game depends on which player wins under optimal play in {the new game}.

\textbf{Case 1:} Suppose that player Y wins {the new game}.
In this case, in the original game, player \B\ occupies
$v_B$ on his first move, and then follows player Y's strategy for {the
  new game}.  At the end of the game, the situation is as in the new
game except that $\A$ has an extra pebble on $v_A$.
 The inequality $\mathrm{dist}(v_B,h)<\mathrm{dist}(v_A,h)$ for all $h\in H$ ensures that
this extra pebble makes no difference for the outcome in $H$:
 player \B\ controls everything in $H$ which was controlled
 by player Y
 at the end of the new game, and ties are preserved in the same way.  Since  $\mathrm{dist}(v_B,k)=\mathrm{dist}(v_A,k)$ for all $k\in
 K$, \B\ gets at least half a vertex for every vertex in $K$ which was
 controlled (scoring $\frac12$)
or tied (scoring $\frac14$)
 by Y at the end of the new game.  Therefore, \B's score of vertices
 within $H\cup K$ 
is 
at least the number of points obtained by Y at the end of the new
game.  Since player Y won the new game, player \B\ must control at
least $|H|/2+|K|/4$ vertices in the original game.

\textbf{Case 2:} Suppose that player X wins {the new game}, or the new game ends in a draw.
In this case, player \B\ plays player X's strategy for {the new game}.
  If player \A\  ever occupies $v_B$ (such a move was not possible for
  Y in
  the new game), then \B\ wastes his following move by playing arbitrarily.
If \B\ ever needs to play on a vertex that he already occupies (from
a previous wasted move), then
he plays arbitrarily again, as in the usual strategy stealing
argument. \B\ also wastes his last move, which was not part of the new game.
At the end of the game, the difference from the situation in the new
game is that (i)~$\A$ has a pebble on $v_A$, (ii)~\A\ has possibly \emph{no
  pebble} on $v_B$, whereas $Y$ had a pebble there, and (iii)~$\B$ has some extra pebbles (in fact, one or two)
from wasted moves. The changes (ii) and (iii) are obviously in $\B$'s favor, hence
it suffices to discuss the effect of~(i). We can also assume that \A\ has
a  pebble on $v_B$, like player~$Y$.

Since $\mathrm{dist}(v_A,h)\geq \mathrm{dist}(v_B,h)$ for all $h\in H\cup K$, 
the additional pebble on $v_A$ has no effect on the outcome for the vertices from
$H\cup K$. Ties remain ties, and vertices under $\B$'s control remain so.
%
Player X had at least $|H|/2+|K|/4$ points at the end of the new game.
It follows that \B\ gets at least this many
vertices under the scoring rules
of the original game,
since the score can only
increase when going back to the original game: for vertices in $K$ is
doubled; for vertices in $H$ it is unchanged.
\end{proof}
\section{There is no lower bound on the Voronoi ratio}\label{NoLowerBoundSection}
The goal of this section is to prove Theorem~\ref{NoLowerBound}.  In fact we prove the following stronger version.
\begin{thm}\label{Zero}
%
  For every $t_0\geq 1$ and $\eps>0$, there is a graph  $G$
  for which player \B\ has a strategy for the Voronoi game ensuring him
  control over at least a fraction $1-\eps$ of the vertices after
  \emph{each} of the rounds $1, \dots, t_0$.
\end{thm}
 This is slightly stronger than Theorem~\ref{NoLowerBound}, which requires for each \emph{fixed} number of
rounds $t \le t_0$,  that a winning strategy exists (possibly a
different strategy for each~$t$).  We will need this stronger
statement when we consider graphs of bounded degree in
Section~\ref{BoundedDegreeSection}. 

\begin{proof}
  We first illustrate the idea for the one-round game ($t=1$).  The
  construction is based on a continuous Voronoi game played on a
  $d$-dimensional regular simplex with the Euclidean metric and $\frac
  1{d+1}$ weight on each vertex.\footnote{We could get rid of the weights by starting a long, narrow path from each vertex of the simplex, giving a construction with uniform weight distribution, but not convex.} In this game, no matter where \A\
  places her pebble, \B\ can take a facet of the simplex 
  that does not contain this pebble and place his pebble on the
  projection of \A 's pebble to the facet. In this way, \A\ gets
  $\frac 1{d+1}$ and \B\ gets $\frac d{d+1}$.

  Consider the point set
$$\{\,(x_1,x_2,\ldots,x_d)\in \mathbb{Z}^d \mid x_i\ge 0,\
x_1+x_2+\dots+x_d = d^2\,\}$$
and connect two points by an edge if their Manhattan distance is 2.
This graph models a regular $(d-1)$-dimensional simplex in $d$
dimensions, and the distances in the graph are $\frac12$ times the $L_1$-distance on
$\mathbb{Z}^d$.
The corners $C$ are the points $(d^2,0,\dots,0)$,
 $(0,d^2,0,\dots,0)$, \dots,
 $(0,0,\dots,d^2)$.
Attach $N$ leaves to each corner.
The distance from $(x_1,\ldots,x_d)$ to the $i$-th corner is $d^2-x_i$.
Suppose $\A$ takes vertex $(x_1,\ldots,x_d)$. Suppose w.l.o.g.\ that $x_1$
is the
largest coordinate. Then $x_1\ge d$, and $\B$ can take
$(x_1-d+1,x_2+1,\ldots,x_d+1)$. This vertex
is closer to all corners except the first.  This ensures that \B\ controls at least $Nd$ vertices, which for sufficiently large $N$ is within $\epsilon$ of $\frac{|G|}{d+1}$.

We now prove the theorem for the general case of $t_0$ moves.  We start with the following set of points.
$$S := \{\,(x_1,x_2,\ldots,x_d)\in \mathbb{Z}^d \mid x_i\ge 0,\
x_1+x_2+\dots+x_d = d^2t_0\,\}$$
As before, we attach $N$ leaves to each vertex.
For a point $x=(x_1,\ldots,x_d)$,
let $\pi_i(x)$ denote the point obtained by subtracting $d-1$ from 
$x_i$ and adding $1$ to all remaining coordinates. This
operation corresponds to projecting $x$ to the simplex facet opposite
the $i$-th corner, except that $x$ is moved only by a fixed step size.
As long as all coordinates of
 $\pi_i(x)$
are nonnegative, moving from $x$ to
 $\pi_i(x)$ is brings us closer to all corners except the $i$-th
 corner.

 Now we try to play against $\A$ as in the case of a single move.  If
 $\A$ takes vertex $(x_1,\ldots,x_d)$, we find the largest coordinate
 $x_i$, and try to move to $\pi_i(x)$. However, this point may already
 be occupied by a previous pebble of~\A. Thus we try the points
 $\pi_i(x)$, $\pi_i(\pi_i(x))$, $\pi_i(\pi_i(\pi_i(x)))$, \dots\ in
 succession.  Since $\A$ has played at most $t_0-1$ previous pebbles,
 one of the first $t_0$ points of this sequence is free, and since
 $x_i\ge dt_0$, it is an element of $S$.

Thus, after each round, \A\ can own at most one additional corner.  If
$\A$ plays one of the $N$ leaves incident to a corner, we can treat
this case as if $\A$ had played the corresponding corner.  Thus, by
making $N$ large enough so that the vertices of $S$ become negligible,
$\A$ will never get more than a fraction $\frac {t_0}{d}+\eps'$ of the
vertices, where $\eps'>0$ can be made as small as we want.
The statement of the theorem follows by setting $d := 1+\lceil t_0/\eps\rceil$.
\end{proof}

\section{Trees} \label{TreesSection}

In this section we investigate the quantity $\VR(T,t)$ when $T$ is a
tree.
We provide tight lower bounds on $\VR( T,t)$ 
for $t=1$ and $t=2$ moves.
For one round, it is well-known that \A\ can always claim half the vertices
of any tree, see for example \cite[Section~6]{TDU}:
\begin{prop}\label{OneRoundTrees}
For all trees $T$, we have
$\VR( T,1)\ge \frac 12.$

This bound is tight, because for 
the path $P_n$ with $n$ vertices, we have
$\VR( P_n,1)\le \frac 12+\frac 1{2n}.$
\end{prop}
\begin{proof}
  The optimal strategy is to put a pebble on a central vertex.  Since
  our proof for two moves will extend the proof of this fact and of
  the existence of central vertices, we include this easy proof here.

An edge of the tree splits it into two parts of size $x\le n/2$ and
$n-x$. We
assign the smaller size $x$ as the \emph{weight} of this edge and
direct it from the smaller side to the larger side.
A tree may have a single undirected edge (of weight $n/2$), which is
called the \emph{central edge} $c_1c_2$.
It is easy to show that every vertex has at most one out-going arc.
There can only be one or two vertices without outgoing arcs (roots). If there
is a single root, it is called 
the \emph{central vertex} $c$ of the
tree; otherwise the two roots are the two vertices of the central
edge.
We can view the tree as a directed tree oriented towards a single root
$c$ or two adjacent roots $c_1,c_2$.

In any tree $T$,
the optimal strategy for \A\ is to play the central vertex or one of
the two vertices incident to the central edge. Removal of this vertex $v$
splits the graph into components of size at most $n/2$. In one
move, \B\ can get at most one component, and 
 thus, \A\ keeps at least half of the vertices.

A path on an even number of vertices, or more generally,
any tree which has a central edge, shows that the bound cannot be improved.
\end{proof}

Combining  Proposition~\ref{OneRoundTrees} with Theorem~\ref{GeneralBound} implies the following.
\begin{cor}
For every tree $T$ and every $t$,
$\VR( T,t)\ge \frac 14$. 
\end{cor}

For the case of two moves, we will improve this lower bound to
$\VR( T,2)\ge \frac 13$, which cannot be improved.
 We need the following lemma:
\begin{lem}\label{crucial}
  Let $T$ be a tree $T$ with $n$ vertices. Either, the central vertex
  $c$ has the following property:
\begin{enumerate}
\item [\rm $C_1$:]
  All components of the graph $T-\{c\}$ have at
most $n/3$ vertices,
\end{enumerate}
or there
are two distinct vertices $u,v$ with the following properties:
\begin{enumerate}
\item [\rm $C_2$:]
  All components of the graph $T-\{u,v\}$ have at
most $n/3$ vertices.
\item [\rm $C_2'$:]
 After removing the edges on the path from $u$ to $v$,
the component $T_u$ containing $u$ and
the component $T_v$ containing $v$ 
contain more than $n/3$ vertices each.
\end{enumerate}
\end{lem}

\begin{proof}
We use the orientation and weight labeling from the proof of
Proposition~\ref
{OneRoundTrees}.
We will try to find our vertices $u$ and $v$ as the vertices
 which have the
following \emph{threshold property}:
\begin{enumerate}
 \item[\rm (i)]
 All incoming edges have weight $\le n/3$.
 \item[\rm (ii)]
 No outgoing edge has weight $\le n/3$.
\end{enumerate}
 Part~(ii) of the condition means generally that
the outgoing edge has weight $>n/3$, but it
includes the case that
there is 
 no outgoing edge at all (the vertex is the
central vertex $c$ or it is incident to the central edge (of weight $n/2$)).
We call a vertex
  with properties {(i)} and~{(ii)} a
 \emph{threshold vertex}.

\begin{claim}\label{threshold-vertex}
  There is at least one {threshold vertex}, and there can be at most
  two threshold vertices.
\end{claim}
\begin{proof}
  To see that a threshold vertex exists, start from a root ($c$ or
  $c_1$ or $c_2$). If it has an incoming edge of weight $>n/3$ proceed
  along this edge, and repeat. Eventually, a threshold vertex must
  be reached.

Since weights are strictly increasing towards the root, no threshold
vertex can be an ancestor of another threshold vertex.
Thus, the subtrees of different threshold vertices must be disjoint.
On the other hand, the subtree rooted at a threshold vertex $u$ must
contain more than $n/3$ vertices:
if $u$ has an outgoing arc, this follows from property~(ii).
If $u$ is the central vertex $c$ or one of the endpoints $c_1,c_2$ of
the central edge, the subtrees have size $n$ and $n/2$ respectively.
It follows that there cannot be more than 2 threshold vertices.
\end{proof}

We note that a tree with a central vertex and two incoming arcs of
weight $>n/3$ must have two threshold vertices, by the argument in the
first part of the proof. We will need this fact later.

Now we can complete the proof of the lemma.
If there is a single threshold vertex $u$ which coincides with the
central vertex $c$, 
all components of
$G-c$ have size $\le n/3$, and we have established condition $C_1$.

Otherwise, there are either (a)~two threshold vertices
$u,v$, or
(b)~a
single threshold vertex $u\ne c$.

Case~(a): There are two threshold vertices $u\ne v$.
Since no threshold vertex is the ancestor of another threshold vertex,
the path from $u$ to $v$ uses the outgoing arc from $u$ (or if $u$
is incident to the central edge, it uses that central edge.) The
weight of this arc is the size of $T_u$, and by the definition of
threshold vertices, it is $>n/3$.  The same argument holds for $v$, and
thus we have established property $C_2'$.

Their are two types of components of $T-\{u,v\}$. There can be an ``inner
component'' that contains the path from $u$ to $v$ (unless $u$ and $v$
are adjacent). The remaining components are the \emph{outer
  components}: they
are connected by
edges that are directed into $u$ and $v$. Again, by the definition
of threshold vertices, their size is $\le n/3$.
The inner component contains everything except $T_u$ and $T_v$, and
hence its size is at most
$|T|-|T_u|-|T_v| < n-n/3-n/3 =n/3$,
thus
 giving property~$C_2$.

Case~(b):
There is a single threshold vertex $u\ne c$.
 In this case, we set
$v:=c$. The path from $u$ to $v=c$ is directed from $u$ to $v$.
As in case~(a), the first edge has weight $>n/3$, and thus
$|T_u|>n/3$.
The last edge is directed towards $v$;
therefore $|T_v| >n/2$, and property $C_2'$ is established.
As above, this implies the bound of $n/3$ on the size of the inner
component. 

The outer components that are incident to $u$ are treated as in
case~(a). Let us consider 
the outer components incident to $v=c$. If there were such a component
with $>n/3$ vertices, it would mean that another threshold vertex
could be found by following this edge down the tree, as we remarked after the proof
of Claim~\ref{threshold-vertex}. This is excluded in case~(b), and
thus we have established
property~$C_2$.
\end{proof}

\begin{thm}\label{TwoRoundTrees}
  \begin{enumerate}
  \item For every tree $T$, $\VR( T,2)> \frac 13$.
  \item For every $\eps>0$ and every $t\ge 2$, there is a tree $T$
 with 
 $\VR( T,t)< \frac 13+\eps$.
  \end{enumerate}
\end{thm}

\begin{proof}
  Lower bound.
  If Lemma~\ref{crucial} produces a single vertex $c$, \A's strategy
  is obvious: take $c$. All components of $T-c$ have size $\le n/3$. With
  two moves, \B\ can take at most 2 components, and thus \A\ keeps at
  least $n/3$ vertices, even without placing her second pebble.

  If Lemma~\ref{crucial} produces two points $u,v$, then \A\ tries to
  put pebbles on them. If this succeeds, we are done: as above, after
  placing two pebbles, \B\ can own at most two components of
  $T-\{u,v\}$, and thus have at most $2n/3$ vertices in total.

  However, \B\ might occupy $u$ or $v$ in his first move. Therefore,
  \A\ has to use a more refined strategy.  Let $T_u$ and $T_v$ denote
  the components of $u$ and $v$ after removing the edges on the path
  between $u$ and $v$. By property $C_2'$, we know that $|T_u|,|T_v| >
  n/3$.  We call the neighbors of $u$ and $v$ that are not on the path
  from $u$ to $v$ the \emph{children} of $u$ and $v$.
  Each child $x$ corresponds to an (outer) component $T_x$ of $T-\{u,v\}$, and we pick
  the child for which this component is largest. Suppose w.l.o.g.\
  that this is a child $u'$ of $u$.
  Then \A\ begins by placing a pebble on $v$. If \B\ does not take $u$
  as a response, \A\ takes it, and we are done, as we have seen above.
  So let us assume that \B\ takes~$u$. Then \A\ takes $u'$ in her
  second move.

Case 1. $\B$ does not take a vertex in $T_v$ in his final move. Then \A\ still owns
$T_v$, and we are done.

Case 2. $\B$  takes a vertex in a component $T_{v'}$, for a child $v'$
of $v$. Then \A\ still owns the rest of $T_v$, excepting $T_{v'}$, plus
all of $T_{u'}$, giving in total at least
$$
|T_v|-|T_{v'}| +|T_{u'}|
\ge
|T_v|> n/3,
$$
by the choice of $u'$.
This concludes the proof of the lower bound.


  Upper bound.
We construct a tree so that  \B\ has a strategy to gain approximately $\frac{2}{3}$ of the vertices for any number of turns $t\geq2$. Observe the following tree and strategy. First we need to introduce a couple of definitions.

   A vertex together with $x$ neighbors of  degree $1$ forms a
   \emph{broom} of size~$x$.
 Take a path and attach  a broom at successive distances
 $1,2,4,8,\dots,2^{m-1}$  from each other.  We call such a path \emph{a
   leg} if it contains $k$ brooms of size $N$. Numbers $k$ and $N$
 will be specified later.  If $N$ is very large, the vertices of the
 path become negligible, and the mass of the graph is concentrated in
 the brooms.

\paragraph{Construction:} 
Take a centre point $c$ that will be of degree three.  We attach two
legs to $c$ and a vertex $h$ forming a broom of size $kN$, which we
call the head,
see Figure~\ref{fig:tree-two-rounds}. If $N$ is large,
 each component of $G-\{c\}$ has about $\frac{1}{3}$ of the vertices.

\begin{figure}[htb]
  \centering
  \includegraphics{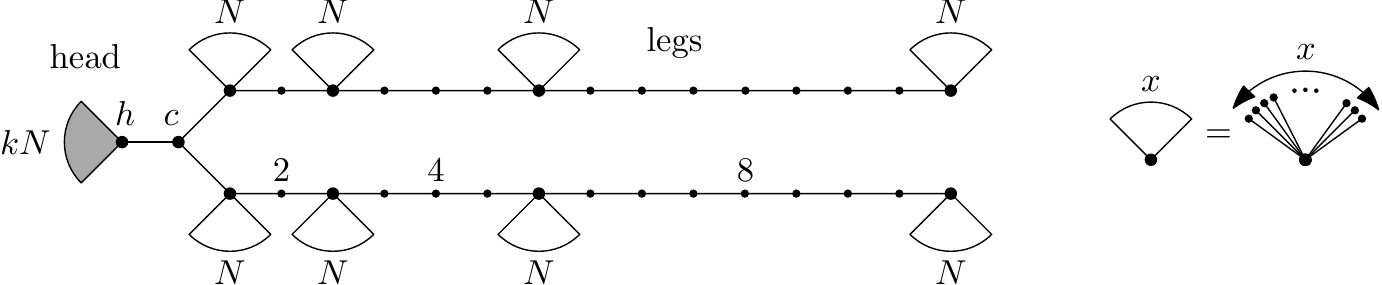}
  \caption{Player $\B$ can ensure to get $2/3-1/k$ in $t=2$ moves.
 A circular sector represents a
    large number $x$ of leaves
    incident to one vertex (a ``broom''). 
 In the example, each leg has $k=4$ brooms.
 }
  \label{fig:tree-two-rounds}
\end{figure}

 The longest path of a leg will be called \emph{the path} of a leg. We define  a natural ordering on the path of a leg.  Vertex $c$ will be on the top and all other \emph{over} and \emph{below} relations of the vertices on  the path of a leg we correlate according to that.

As a straightforward consequence of the exponentially increasing
distances of a leg, we obtain the following observation.

\begin{obs}\label{legs}
 Suppose player \B\ claims a vertex $v$ on the
  path of a leg and $w$ is the closest vertex below $v$ such that $w$
  or a leaf adjacent to $w$ is claimed by a player (either by \A\ or
  \B).  Then \B\ controls all the brooms lying below $v$ up to and
  including $w$, except at most one.
\end{obs}
As a consequence, we get.
\begin{obs}\label{dominate}
 Suppose player \B\ has claimed $c$ or the highest vertex of a leg $l$ and
 player \A\ has claimed $i$ vertices of this leg.
In addition, suppose
 for each vertex 
 $w$ on the path for which $w$
  or a leaf adjacent to $w$ is claimed by player \A,
some player has claimed the vertex $w'$ immediately below it (unless
$w$ is the lowest vertex, for which $w'$ does not exist).
Then \A\ owns at most $i$ brooms, plus possibly $i$ individual leaves
in brooms which are otherwise taken by \B.  
\end{obs}
If this condition is fulfilled, we say that \B\ \emph{dominates} the
leg.
When \B\ has claimed $c$ or the highest vertex of a leg $l$, then
he can ensure that he dominates $l$ if he can place as many pebbles
into
$l$ as $\A$, in addition to the pebble placed at $c$ or the highest
vertex, by following the strategy suggested by
Observation~\ref{dominate}.

\paragraph{Strategy:} 
\begin{itemize}
\item If \A\ takes the centre vertex $c$ in the first turn, then \B\
  takes $h$. In the second turn \A\ can either take a leaf from the
  broom at $h$ or a vertex from one of the legs. In either case there
  is a leg $l$ where \A\ did not put any pebble yet. In his second
  turn \B\ takes the closest vertex to $c$ on $l$. Therefore at this
  point of the game \B\ owns the whole leg $l$ completely. In his further moves,
  \B\ will \emph{defend} $l$, ensuring that he dominates $l$ according
  to the condition of Observation~\ref{dominate}:
If \A\ claims a vertex $v$ on the path of $l$, then \B\ claims the vertex below
$v$ if it is defined and available. If it is not available, then $v$
is either the lowest vertex or it is above an already claimed
vertex. In either case \B\ can claim any available vertex.
 If 
\A\
 claimed a leaf belonging to a broom on $l$,
 \B\ claims the neighbour of $v$ on the path of $l$ if it is
 available. Otherwise, \B\ can claim any available vertex. If \A\
 claims a vertex not belonging to $l$, then \B\ claims any available
 vertex.

\item If \A\ does not take the centre vertex $c$, then \B\ takes it.
From now on, \B\ will try to defend both legs, as in the
  strategy above.
The problem is that \B\ may be one move short in his defensive
strategy, if \A\ has moved to a leg 
in his first move.
If \A\ takes a vertex from the head in any of her turns (including her
first move), then \B\ can catch up with
\A\ and dominate both legs from then on.
If \A\ never takes a vertex from the head, then \B\ can successfully defend only
one leg, but he owns the whole head.

Now we describe the strategy more precisely.
There are two possibilities.
Suppose \A\ takes a vertex from
  the broom formed by $h$ in the first turn. Then \B\ claims
  $c$, and in 
 all his forthcoming turns, \B\ will defend both legs, see the
  strategy above.
 More precisely, when \A\ claims a vertex from the
  leg $l_1$, then \B\ defends $l_1$.
 When \A\ claims a vertex from the
  leg $l_2$, then \B\ defends $l_2$. 

Consider the other case, when \A\ claims a vertex from a
  leg in her first turn. Then \B\ claims the centre $c$, and
in all
  remaining turns, if \A\ claims a vertex from a leg, \B\ will
  defend that leg.
 If in a turn \A\ claims a vertex from the broom
  formed by $h$, then \B\ will claim the vertex which is right below
  the vertex taken by \A\ in her first turn if it is defined and
  available.
 If the taken vertex by \A\ in her first turn was a leaf
  in a broom, \B\ takes the broom if available. In all other cases,
  \B\ is free to chose any available vertex.

\end{itemize}

\paragraph{Analysis of the strategy:} 
In the first case \B's strategy was to gain $h$ and as much as possible from a leg. As a result of this strategy, by Observation~\ref{dominate}, \B\ ensures himself the whole leg except of those brooms in which \A\ claimed a vertex. In the end of the game, \B\ will control all vertices of the broom formed by $h$ except at most $t$ leaves and the leg $l$ without at most $t$ brooms.

When  \B's strategy was to defend both legs by Observation~\ref{dominate} he ensures himself both legs except those brooms in which \A\ claimed a vertex, which is at most $t$. 

In the last case \B's strategy was to defend a leg, while he controls the large broom of $h$, and if \A\ claimed a vertex from the broom formed by $h$, then \B\ defended both legs. Thus either \B\ obtained both legs except of those brooms in which \A\ claimed a vertex, or \B\ gained the broom formed by $h$ and a leg possibly without at most $t$ brooms.

\paragraph{Counting the gain:}

By our construction the tree contains $kN$ vertices in the brooms of each of the three subtrees connected to the centre vertex $c$. Hence, there are $3kN$ vertices in the brooms of the tree.
In each of the three cases \B\ gains at least $2kN-tN$ vertices
in brooms. There are more vertices in the tree outside the
brooms but we achieve  that the number of those is negligible by
increasing~$N$. Therefore, \B\ gets
$\frac{2}{3}-\frac{t}{3k}$, and for big $k$ this amount is
close to $\frac{2}{3}$. Hence, the statement of the theorem follows.
\end{proof}

\section{Graphs with bounded degree}\label{BoundedDegreeSection}
\label{max-degree}
In this section, we investigate when player \B\ is able to obtain some positive proportion of the vertices, i.e., for a fixed $\eps>0$ we are interested in knowing for which graphs $G$ we have $\VR(G,t)\leq 1-\eps$.  For every $\eps>0$ and $t$, there are certainly graphs for which $\VR(G,t)>1-\eps$. For example, we could take $G$ to be a star with more than $\frac{t}{\eps}$ leaves.  However if $G$ is not allowed to have vertices of high degree, then the situation changes.  

\begin{lem}\label{MaxDegreeOneRound}
In a connected graph $G$ with $n$ vertices and maximum degree $\Delta$, we have
$$\VR(G,1)\leq 1-\frac{1}{\Delta}+\frac{1}{n\Delta}.$$
\end{lem}
\begin{proof}
Let $v$ be the vertex chosen by player \A\  on her first move,
and let $x_1,\dots,x_k$ be the neighbors of $v$, with $k\le \Delta$. Let $H(x_i)$ be the set of vertices which are closer to $x_i$ than to $v$. Obviously every vertex of $G$ belongs to at least one $H(x_i)$.
\B\ picks the neighbor $x$ 
for which $|H(x)|$ is largest and will control
at least $|H(x)|\ge (n-1)/\Delta = n/\Delta -1/\Delta
$ vertices.
This implies $\VR(G,1)\leq 1-\frac{1}{\Delta}+\frac 1 {n\Delta}$.
\end{proof}

Combining Lemma~\ref{MaxDegreeOneRound} with
Theorem~\ref{GeneralBound} we obtain the following.
\begin{cor}\label{MaxDegreeManyRounds}
In a connected graph $G$ with $n$ vertices and maximum degree $\Delta$, we have
$$\VR(G,t)\leq 1-\frac{1}{2\Delta}+\frac{1}{2n\Delta}
.$$
\end{cor}

Let $S_{k,N}$ is the graph formed from a star with $k$ leaves by replacing every leaf with a path of length $N$.  Since player \A\  can always choose the center of the star on her first move, it is easy to see that  $\VR(S_{k,N},1)\to 1-\frac{1}{k}$ as $N\to \infty$. This shows the bound in Lemma~\ref{MaxDegreeOneRound} cannot be substantially improved.  

For $t\geq 2$, we were not able to determine whether the bound in Corollary~\ref{MaxDegreeManyRounds} can be improved or not.  However we were able to find graphs which show that the bound in Corollary~\ref{MaxDegreeManyRounds} cannot be increased by more than $\frac{1}{\Delta}$, by proving the following.

\begin{thm}\label{MaxDegreeLowerBound}
For every $\Delta,t\geq 1$ and $\eps>0$,  there is a connected graph $G$ with maximum degree $\Delta$ satisfying 
$$\VR(G,t)\geq 1-\frac{1}{\Delta}-\eps.$$
\end{thm}

In order to prove Theorem~\ref{MaxDegreeLowerBound}, we first need to show that for every  $t$, there are graphs with maximum degree $3$ on which \B\ can claim almost all the vertices after $t$ rounds.  We prove the following.

\begin{lem}\label{OneRoundBoundedDegree}
For every  $t\geq 1$ and $\eps>0$, there is a graph $G_{t,\eps}$  with maximum degree $3$ and  the following property:  Player \B\ has a strategy for the Voronoi game on $G_{t,\eps}$ such that after each round $1, \dots, t$, he will control a fraction $1-\eps$ of the vertices after
  \emph{each} of the rounds $1, \dots, t$.
\end{lem}
\begin{proof}
The proof is an extension of Theorem~\ref{Zero}.  We set $d=\lceil\frac{2t}{\eps}\rceil$.
Instead of a hyperplane in $\Z^d$, we will take a full cube of side
length $L=d^2t$ from which the lowest corner has been cut off:
 the graph $H$ has vertex set 
$$\{\,(x_1,x_2,\ldots,x_d)\in \mathbb{Z}^d \mid 0\le x_i
\le L,\ 
x_1+x_2+\dots+x_d\ge L
\,\}.$$
Two vertices are connected in $H$ whenever their $L_1$ distance is
$1$, i.e. they differ in one coordinate and the difference is $1$. Then the distance between any two vertices equals their $L_1$ distance.
As before,
the corners $C$ are the points $(L,0,\dots,0)$,
 $(0,L,0,\dots,0)$, \dots, $(0,\dots,0,L)$.
The distance
from a vertex $(x_1,\ldots,x_d)$
 to the $j$-th corner can be calculated as
$$L+\sum_{i=1}^d x_i -2 x_j.$$

The strategy of Theorem~\ref{Zero}
 must be adapted to
  account for the fact that $H$ has additional vertices:
Suppose $\A$ takes vertex $x=(x_1,\ldots,x_d)$, and suppose w.l.o.g.\
that $x_1\ge L/d = dt$
is the
largest coordinate.
 Then 
 $\B$ calculates the response point
 $\pi_1(x)=(x'_1,x_2',\ldots,x_d')$,
where $x_i' = \min\{x_i+1,L\}$ for $i=2,\ldots,d$, and
$x_1' = x_1-(d-1) \ge (t-1)d$.
These formulas ensure that $x'\in H$, and 
one can easily show that
every corner except the first is closer to $x'$ than to $x$.

By an argument analogous to the proof of Theorem~\ref{Zero}, one can
find a
 strategy $\mathcal{S}$ for \B\ in the 
 Voronoi game on $H$
which ensures that after each round $1, \dots, t$, there are at least
$d-t$ corners $c$ satisfying 
\begin{equation}
\nonumber 
\dist(c,B)<\dist(c,A)  ,
\end{equation}
 where
$A$ and $B$ are the sets of vertices chosen by \A\ and \B\
respectively. 


To cut down the maximum degree, we use a variation of the
\emph{cube-connected cycles} of Preparata and Vuillemin~\cite{PV}.
We construct ``grid-connected cycles''.
 \begin{figure}[htb]
   \centering
   \includegraphics{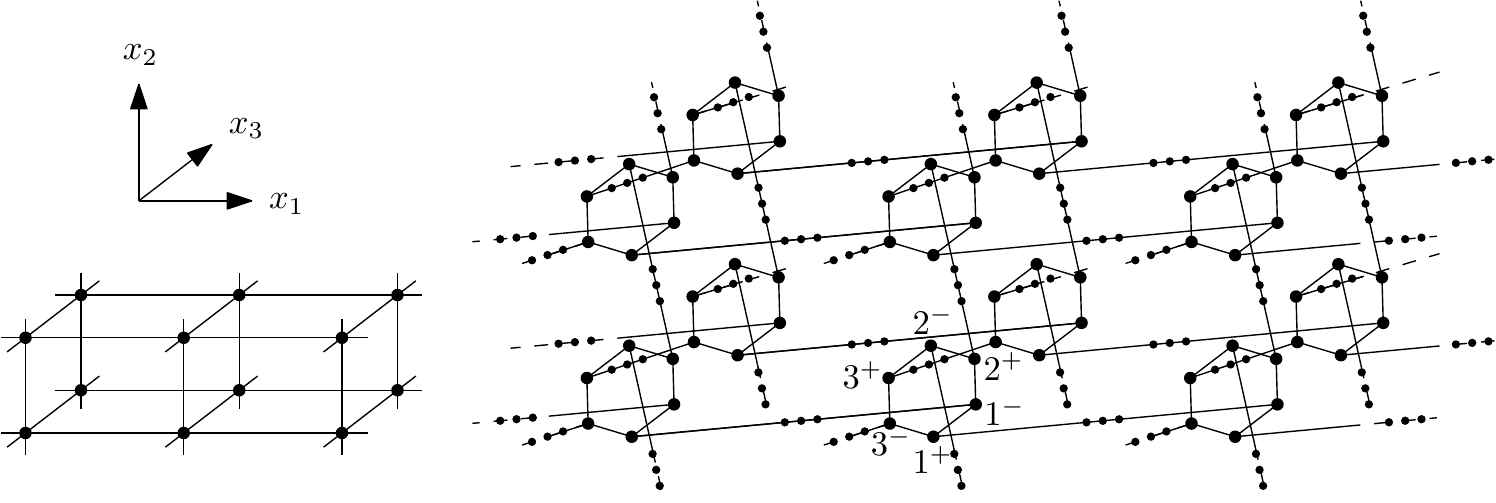}
   \caption{Representing the a $3\times2\times2$ section of the grid
     $\Z^3$ (shown on the left) by grid-connected cycles. The
     connection paths should have $6d-2=16$ intermediate vertices, but only 3
   are shown.}
   \label{fig:cube-connected}
 \end{figure}
Each point $x\in\Z^d$ is replaced by a circular ring of $2d$ nodes
that are labeled $1^+,1^-,2^+,2^-,\dots,d^+,d^-$ in cyclic order.
We denote them by $x(1^+)$, $x(1^-)$ etc.
For $i=1,\ldots,d$,
 node $x(i^+)$ is connected to
 node $x' (i^-)$ by a \emph{connection path} of $6d-1$ edges, where
 $x'=x+e_i$
and
$e_i$ is the $i$-th unit vector.
Figure~\ref {fig:cube-connected} shows a three-dimensional example.
The resulting graph has maximum degree~3.  Nodes on the
boundary of the cube have unused connections. 
For each corner, we pick one of these degree-2 nodes
and attach a very long path of length $N$ to it.
As usual, we make $N$ so big that the 
 original
graph becomes only a negligible fraction of the whole graph.
This gives us the graph $G = G_{t,\eps}$.

 The following proposition implies that playing the Voronoi game on
 $G_{t,\eps}$ is approximately the same as playing it on $H$.
The distances are preserved up to a multiplicative factor with an additive error.
 \begin{prop}
  Let $ x(p^\pm)$ and $y(q^\pm)$ be two vertices of $G$ corresponding
  to grid points $x,y\in \Z^d$. Then their distance
$\mathrm{dist}(x(p^\pm), y(q^\pm))$ in the graph is bounded as follows:
  \begin{equation*}
6d \cdot\lVert x-y\rVert_1 -1 \le
\mathrm{dist}(x(p^\pm), y(q^\pm)) \le 6d\cdot
\lVert x-y\rVert_1 +5d
  \end{equation*}
 \end{prop}
 \begin{proof}
    Lower bound.
 The connection paths that connect
different rings correspond to neighbouring points in $\Z^d$. Hence the
path between
 $ x(p^\pm)$ and $y(q^\pm)$ 
needs at least
$\lVert x-y\rVert_1$ of these connection paths. But since these paths are not directly adjacent, a path in $G$ has has to contain at least one ring edge between any
two connection paths.

Upper bound.
Consider two nodes
 $ x(p^\pm)$ and $y(q^\pm)$ that we want to connect by a path.
Let $u=(u_1,\dots,u_d)$ be the elementwise maximum of $x$ and $y$:
$u_i=\max\{x_i,y_i\}$.
Then we have
$\lVert x-y\rVert_1 =
\lVert x-u\rVert_1
+\lVert y-u\rVert_1$
 To get from 
 $ x(p^\pm)$ to $y(q^\pm)$, we go via the ring $u$.
We connect
 $ x(p^\pm)$ to the node $u((p-1)^-)$
by sequentially increasing each coordinate value
 $i=p,p+1,\ldots, p-1$
  from $x_i$ to $u_i$.
This procedure works because the graph represents a
subcube of $\Z^d$, from which some ``lower'' part has been removed. It
is always possibly to increase a coordinate, up the maximum~$L$.

We make possibly one initial step to $x(p^+)$.
The coordinate move in direction $i$
goes from 
some node $z(i^+)$ to some node $z'(i^-)$
in $6d|x_i-u_i|-1$ steps, 
strictly
alternating between connection paths and ring edges.
One more step brings us to
$z'((i+1)^+)$ to get ready for the next coordinate direction,
for a total of
 $6d|x_i-u_i|$ steps.
This bound does not work for $x_i=u_i$: there we need 2 steps from $z(i^+)$ to
$z((i+1)^+)$.
In total, we can bound the number of steps to at most
$\sum_{i=1} ^d (6d|x_i-u_i|+ 2)
= 6d\lVert x-u\rVert_1 +2d$.
Similarly, $y(q^\pm)$ is connected to some vertex on the ring $u$ in
at most
 $6d\lVert y-u\rVert_1 +2d$, steps, and we need at most $d$ additional
 steps on the ring~$u$.
 \end{proof}



We continue the proof of Lemma~\ref{OneRoundBoundedDegree}.
Let $V'\subset V(G)$ denote the nodes on rings, and
 let $f\colon V'\to H\subset \Z^d$ denote the function which maps
 every node $x(i^\pm)$ to its grid point $x$. 
As a consequence of the previous proposition, for two vertices
 $u,v\in H$, we can recover the $L_1$ distance of their corresponding
 grid points from their distance in the graph:
  \begin{equation}
\nonumber
\lVert f(u)-f(v)\rVert_1
 = \left\lfloor \frac{\mathrm{dist}(u,v) + 1}{6d}
\right\rfloor
  \end{equation}
This means that strict equalities between distances in $H$ carry over
to corresponding vertices of $V'$.

Player \B's strategy on $G_{t,\eps}$ is as follows:
If \A\ moves to a node $u$ on one of the rings, \B\ interprets this as a move to $f(u)$ in
$H$,
calculates
 his response
 $x$ according to the
strategy
$\mathcal{S}$ on $H$, 
and chooses
an arbitrary node $x(q^\pm)$ on the corresponding ring.
If \A\ selects several nodes on the same ring, they are interpreted
as wasted moves in $H$.
If \A\ plays on one of the long paths, this is interpreted as a move
to the corresponding corner vertex.
Finally, \A\ might move to a node $w$ on a connection path between
vertices $u$ and $u'$ on two rings. Then
 $f(u)$ and $f(u')$ differ in exactly one coordinate $x_j$. Let us
 assume that $f(u)$ has the smaller $x_j$-coordinate. Then $\B$ interprets this
 as a move to $f(u)$ and responds as above. To analyze the error
incurred 
 by this interpretation, let us imagine that \A\ had covered \emph{both} $u$ and
 $u'$. This would certainly be more advantageous for \A\ than
 covering~$w$ alone. However there is
 only one corner which is closer to $f(u')$ than to $f(u)$: the
 $j$-th corner. All other corners are closer to $f(u)$. Thus, by
 allowing
\A\ to cover the vertex $u'$ in addition to  $u$, she can 
 win at most one 
additional corner. It follows that after the $k$-th round ($1\le k\le
t$), \A\ owns at most $2k$ nodes $c\in
G_{t,\eps}$ to which a path of length $N$ is attached.
This implies the lemma.
%
\end{proof}

We can now prove Theorem~\ref{MaxDegreeLowerBound}.
\begin{proof}[Proof of Theorem~\ref{MaxDegreeLowerBound}.]
For given $\eps$ and $t$, we construct the graph $G$
 from $\Delta$ disjoint copies of $G_{t,\eps}$ called $G_1, \dots, G_\Delta$ and an extra vertex $v$ by adding exactly one edge between $v$ and $G_i$ for each $i$.

On her first move \A\  claims the vertex $v$. Subsequently \A\ always claims a vertex from the same $G_i$ as \B\ in the previous move. \A\
treats $G_1,\dots, G_\Delta$ as separate games, and plays the strategy of
the second player given by Lemma~\ref{OneRoundBoundedDegree}. Hence she controls at least $(1-\epsilon)|V(G_i)|$ vertices of $G_i$ after her move. However, she cannot answer the very last move of \B\.e

This ensures that at the end she controls at least $(1-\eps)|V(G_i)|$ of the vertices of each $G_i$ except for one.  Since for $i\neq j$ there are no edges between $G_i$ and $G_j$, Player \B\ can capture at most $|G_i|=\frac{1}{\Delta}(|V(G)|-1)$ vertices on his last move.  Therefore \A\  controls at least at least $\left(1-\eps-\frac{1}{\Delta}\right)|V(G)|$ vertices at the end of the game, proving the result.
\end{proof}

\subsection*{Remarks and acknowledgment}
Several questions are left open.
Are there trees $T$ for which $\VR(T,t)$ is close to $\frac 14$ or can the
first player always get at least $\frac 13$ of the vertices, for $t\ge
3$?
How much can she get if they play on a planar graph? What about biased
versions of the game, where the players play different amounts of pebbles?

This work started at the Fourth Eml\'ekt\'abla Workshop in Tihany in
August 2012,
whose topic was Positional Games. We are thankful to the organizer and
 Milo\v s Stojakovi\'c for posing the problem.
We are grateful to Younjin Kim and Tam\'as Hubai for several observations.

\bibliography{Voronoi}
\bibliographystyle{abbrv}

\end{document}